\documentclass[amstex,12pt,reqno]{amsart}

\usepackage{amsmath,amsfonts,amssymb,amsthm,enumerate,hyperref,multicol}

\usepackage{graphicx}

\newtheorem{lemma}{Lemma}
\newtheorem{thm}{Theorem}

\newtheorem{corollary}{Corollary}
\newtheorem{remark}{Remark}

\numberwithin{equation}{section}

\title[]
{Certain Approximation Results for Kantorovich Exponential Sampling Series} 



\author{Shivam Bajpeyi}

\address{School of Mathematics, Indian Institute of Science Education and Research, Thiruvananthapuram, India.}

\email{shivambajpai1010@gmail.com}

\author{A. Sathish Kumar}

\address{Department of Mathematics, Indian Institute of Technology Madras, Chennai-600036, India}

\email{sathishkumar@iitm.ac.in, mathsatish9@gmail.com}

\author{P. Devaraj}

\address{School of Mathematics, Indian Institute of Science Education and Research, Thiruvananthapuram, India.}

\email{devarajp@iisertvm.ac.in}

\begin{document}

\begin{abstract}
In this paper, we study a strong inverse approximation theorem and saturation order for the family of Kantorovich exponential sampling operators. The class of log-uniformly continuous and bounded functions, and class of log-H\"{o}lderian functions are considered to derive these results. We also prove some auxiliary results including Voronovskaya type theorem, and a relation between the Kantorovich exponential sampling series and the generalized exponential sampling series, to achieve the desired plan. Moreover, some examples of kernels satisfying the conditions, which are assumed in the hypotheses of our theorems, are discussed. 
\end{abstract}

\maketitle

\section{Introduction}
The problem of sampling and reconstruction of functions is a fundamental aspect of approximation theory, with important applications in signal analysis and image processing (\cite{apl1,apl2}). A significant breakthrough in sampling and reconstruction theory was collectively achieved by Whittaker-Kotelnikov-Shannon. They established that any band-limited signal $f$, i.e. the Fourier transform of $f$ is compactly supported, can be completely recovered using its regularly spaced sample values (see \cite{srvbut}). This result is widely known as \textit{WKS sampling theorem}. Butzer and Stens \cite{but1} generalized this result significantly for not-necessarily band-limited signals. Since then, several mathematicians have been making significant advancements in this direction, see \cite{butzer2,k2007,tam,in3,tun1}. 
\par
The problem of approximating functions with their exponentially-spaced sample values can be traced back to the work of Ostrowski et.al. \cite{ostrowsky}, Bertero and Pike \cite{bertero}, and Gori \cite{gori}. In order to deal with exponentially-spaced data, they provided a series representation for the class of Mellin band-limited functions (defined in Section \ref{section2}). This reconstruction formula is referred as the \textit{exponential sampling formula} and defined as follows. For $f:\mathbb{R}^{+} \rightarrow \mathbb{C}$ and $c \in \mathbb{R},$ the exponential sampling formula is given by (see \cite{butzer3})
\begin{equation} \label{expformula}
(E_{c,T}f)(x):= \sum_{k=-\infty}^{\infty} lin_{\frac{c}{T}}(e^{-k}x^{T}) f(e^{\frac{k}{T}})
\end{equation}
where $lin_{c}(x)= \dfrac{x^{-c}}{2\pi i} \dfrac{x^{\pi i}-x^{-\pi i}}{\log c} = x^{-c} sinc(\log x)$ with continuous extension $lin_{c}(1)=1.$ Moreover, if $f$ is Mellin band-limited to $[-T,T],$ then $(E_{c,T}f)(x)=f(x)$ for each $x \in \mathbb{R}^{+}.$
\par 
The exponentially spaced data can be observed in various problems emerging in optical physics and engineering, for example Fraunhofer diffraction, polydispersity
analysis by photon correlation spectroscopy, neuron scattering, radio
astronomy etc (see \cite{casasent,ostrowsky,bertero,gori}). Therefore, it became crucial to examine the extensions and variations of the exponential sampling formula \eqref{expformula}.
Butzer and Jansche \cite{butzer5} investigated into the exponential sampling formula, incorporating the analytical tools of Mellin analysis. They established that the theory of Mellin transform provides a suitable framework to handle sampling and approximation problem related to exponentially-spaced data. The foundational work on the Mellin transform theory was initially undertaken by Mamedov \cite{mamedeo}. Subsequently, Butzer and his colleagues made significant contributions to the field of Mellin theory in \cite{butzer3,butzer5}. For some notable developments on Mellin theory, we refer to \cite{bardaro1,bardaro9,bardaro2,bardaro3} etc. In order to approximate a function which is not necessarily Mellin band-limited, the theory of exponential sampling formula \eqref{expformula} was extended in \cite{bardaro7} using generalized kernel satisfying suitable conditions. This gives a method to approximate the class of log-continuous functions by employing its exponentially spaced sample values. For $x \in \mathbb{R}^{+}$ and $w>0,$ the generalized exponential sampling series is given by (see \cite{bardaro7})
\begin{equation} \label{genexp}
(S_{w}^{\chi}f)(x)= \sum_{k=- \infty}^{\infty} \chi(e^{-k} x^{w}) f( e^{\frac{k}{w}})
\end{equation}
for any $ f: \mathbb{R}^{+} \rightarrow \mathbb{R}$ such that the series \eqref{genexp} converges absolutely. Various approximation properties associated with the family of operators \eqref{genexp} can be observed in \cite{comboexp,bardaro11,bevi,diskant}. The approximation properties of exponential sampling operators based on artificial neural network can be found in \cite{sn,self}. In order to approximate integrable functions, the series \eqref{genexp} is not suitable. To overcome with this, the following Kantorovich type modification of the family \eqref{genexp} was studied in \cite{own}. For $ x\in \mathbb{R}^{+}, k \in \mathbb{Z}$ and $w>0,$ the Kantorovich  exponential sampling series is defined by
\begin{equation} \label{kant}
(I_{w}^{\chi}f)(x):= \sum_{k= - \infty}^{\infty} \chi(e^{-k} x^{w})\  w \int_{\frac{k}{w}}^{\frac{k+1}{w}} f(e^{u})\  du \ \
\end{equation}
whenever the series (\ref{kant}) is absolutely convergent for any locally integrable function $ f: \mathbb{R}^{+} \rightarrow \mathbb{R}.$ Some direct and inverse approximation results for the family $(I_{w}^{\chi}f)$ have been discussed in \cite{own,lnr} which includes basic convergence theorem, higher order asymptotic convergence result and quantitative approximation theorem. Also, an inverse approximation theorem in case of $f \in \mathcal{C}^{(1)}(\mathbb{R}^{+})$ was proved in \cite{own} under the assumption that the fist order moment vanishes on $\mathbb{R}^{+}.$ For some recent advancements related to the family \eqref{kant}, we refer to \cite{SKD,aral,kursun,prashant}.
\par 
In the present work, we deduce a strong inverse approximation result for the family of Kantorovich exponential sampling operator $(I_{w}^{\chi})$ for $f \in \mathcal{C}(\mathbb{R}^{+}),$ without assuming that first order algebraic moment vanishes. This not only broadens the underlying class of functions but also enable the application of our theory to some other kernels, for instance, the class of Mellin B-spline kernels (see Section \ref{section4}). We also establish the saturation order, i.e. the highest order of convergence that can be achieved, for $(I_{w}^{\chi}f)$ in case of $f \in \mathcal{C}(\mathbb{R}^{+}).$ The problem of saturation order for the family of operators $(I_{w}^{\chi}),w>0$ is to find a suitable class $\mathcal{F}$ of real valued functions defined on $\mathbb{R}^+,$ a subclass $\mathcal{S}$ and a positive non-increasing function $\rho(w),w>0$ satisfying the following: there exists $h\in \mathcal{F}\setminus\mathcal{S}$ with $\|I_{w}^{\chi}h - h\|= \mathcal{O}(\rho(w))$ as $w \rightarrow \infty$ and whenever $f\in \mathcal{F}$ with $\|I_{w}^{\chi}f - f\|= {o}(\rho(w))$ as $w \rightarrow \infty$ implies that $f\in \mathcal{S}$ and vice versa. Several authors have investigated the inverse approximation results and saturation order for various sampling operators significantly, see \cite{in1,in2,in3,costainverse,costaFourier} etc. 
\par 
The proposed plan of the paper is as follows. In order to derive these results, we first define an appropriate average type kernel and derive some auxiliary results mainly concerned with this new kernel in Section \ref{section2}. In Section \ref{section3}, we establish a relation between the operator (\ref{kant}) and the derivative of the operator (\ref{genexp}) based on average type kernel. Further, we derive the asymptotic formula for the operator (\ref{kant}) using Mellin Taylor formula. By using these results, we prove the saturation theorem and inverse result for the family of sampling operators (\ref{kant}). In Section \ref{section4}, we discuss some examples of kernels satisfying the conditions, which are assumed in the hypotheses of the theorems. 
\section{Preliminaries and Auxiliary Results}\label{section2}
Let $\mathbb{R}^+$ be the set of positive real numbers and $L^{p}(\mathbb{R}^+),\ 1 \leq p < \infty$ consists of all p-integrable functions in the Lebesgue sense on $\mathbb{R}^+$ with usual $p-$norm. Further $L^{\infty}(\mathbb{R}^+)$ denotes the class of bounded measurable functions defined on $\mathbb{R}^+$ with $\|.\|_{\infty}$ norm. Let $X_{c}$ be the space of functions $f: \mathbb{R}^+ \rightarrow \mathbb{R}$ such that $f(\cdot) (\cdot)^{c-1} \in L^{1}(\mathbb{R}^+)$ for some $c \in \mathbb{R},$ equipped with following norm
$$\Vert f \Vert_{X_c} = \int_0^{\infty} |f(t)| \ t^{c} \frac{dt}{t}.$$
For $f \in X_{c},$ the Mellin transform of $f$ is given by
$$[f]^{\wedge}_M (s):= \int_0^{\infty} f(t)\ t^{s}\frac{dt}{t},\ \ \ (s = c + ix, \ x \in \mathbb{R}).$$ 
One can observe that the Mellin transform is well defined in $X_{c}$ as a Lebesgue integral. Further, for $c,t \in \mathbb{R}$ and $T>0,$ any function $f \in X_{c}(\mathbb{R}^{+})$ is said to be Mellin band-limited to $[-T,T],$ if $[f]_{\hat{M}}(c+it)=0$ for $|t| > T.$ For more details on theory of Mellin transform, we refer to \cite{butzer3,butzer4}.
The point-wise Mellin derivative of the function $f$ is defined by the following limit
$$ \theta_{c} f(t)= \lim_{h \rightarrow 1} \frac{\tau_{h}^{c}f(t) - f(t)}{h-1} = t f^{'}(t)+c f(t)\ ,$$ provided $f^{'}$ exists, where $\tau_{h}^{c}$ is the Mellin translation operator $(\tau_{h}^{c}f)(t):= h^{c} f(ht).$ 
Furthermore, the Mellin differential operator of order $r$ is given by $\theta_c^r:=\theta_c(\theta_c^{r-1}).$ Throughout this paper, we consider $\theta_c:=\theta_c^1$ and $\theta f:=\theta_{0} f.$
\par 
We now give the definition of recurrent function. We say that a function $f :\mathbb{R}^+\rightarrow\mathbb{C}$ is recurrent if $f(x)=f(e^{a}x),$ $\forall$ $x\in \mathbb{R}^+$ and for some $a\in \mathbb{R}$ (see \cite{recurrent}). The fundamental interval of the above recurrent functions can be taken as $[1, e^a].$
\par 
Let $C(\mathbb{R}^+)$ denotes the space of all uniformly continuous and bounded functions on $\mathbb{R}^+$ with norm $\|f\|_{\infty} := \sup_{t \in \mathbb{R}^+} |f(t)|.$ For any $\nu \in \mathbb{N},$ $C^{(\nu)}(\mathbb{R}^+)$ be the subspace of $C(\mathbb{R}^+)$ such that $f^{(r)} \in C(\mathbb{R}^+)$ for each $r \leq \nu, r \in \mathbb{N}.$  Also, $C_{c}^{\infty}(\mathbb{R}^{+})$ represents the space of all infinitely differentiable functions which are compactly supported in $\mathbb{R}^{+}.$ A function $f: \mathbb{R}^+ \rightarrow \mathbb{R}$ is said to be {log-uniformly continuous} on $\mathbb{R}^+$ if $\forall$\ $\epsilon > 0,$ there exists a $\delta > 0$ such that $|f(x) -f(y)| < \epsilon$ whenever $| \log x - \log y| < \delta,$ for any $x, y \in \mathbb{R}^{+}.$ 
Further, $\mathcal{C} (\mathbb{R}^+)$ denotes the space of all log-uniformly continuous and bounded functions defined on $\mathbb{R}^{+}.$ Analogous to the classical case, for any $\nu \in \mathbb{N},$ $\mathcal{C}^{(\nu)}(\mathbb{R}^+)$ be the subspace of $\mathcal{C}(\mathbb{R}^+)$ such that $(\theta^{r}f) \in \mathcal{C}(\mathbb{R}^+)$ for each $r \leq \nu, r \in \mathbb{N}.$ For $f \in C^{(r)}(\mathbb{R}^{+}),$ the Mellin's Taylor formula is defined by (see \cite{bardarotaylor})
$$ f(tx)=f(x)+ (\theta f)(x) \log t + \frac{(\theta ^{2}f)(x)}{2!} \log ^{2} t + \cdots + \frac{(\theta ^{n}f)(x)}{n!} \log ^{n} t + h(x) \log ^{n} t \ ,$$
where $h: \mathbb{R}^{+} \rightarrow \mathbb{R}$ is bounded and $h(x) \rightarrow 0$ as $x \rightarrow 1.$\\

A continuous function $ \chi :\mathbb{R}^{+} \rightarrow \mathbb{R}$ is said to be kernel if it fulfils the following conditions:
\begin{itemize}
\item[$(\chi_{1})$] For any $ u \in \mathbb{R}^{+},\ $
$  \displaystyle \sum_{k=- \infty}^{\infty} \chi(e^{-k} u) =1, \hspace{0.2cm} \mbox{uniformly on}\ \mathbb{R}^+.$

\item[$(\chi_{2})$] $\displaystyle m_{1}(\chi,u) := \sum_{k=- \infty}^{\infty} \chi(e^{-k} u) (k - \log u) \ :=  m_{1}^{\chi} \in \mathbb{R} .$

\item[$(\chi_{3})$] For some $\beta \geq 1 ,$
$\displaystyle M_{\beta}(\chi):= \sup_{u \in \mathbb{R}^+} \sum_{k= - \infty}^{\infty}  |\chi(e^{-k} u)| |k- \log u|^{\beta}< \infty .$

\item[$(\chi_{4})$] For every $\gamma>0,$ 
$ \displaystyle\lim_{w\rightarrow \infty} \sum_{|w\log x-k|>w\gamma} |\chi(e^{-k} x^w)| \ |w\log x-k|=0 $ uniformly on $\mathbb{R}^{+}.$
\end{itemize}

\begin{remark} \cite{own}
One can deduce that for $\alpha, \beta \in \mathbb{N}_{0}$ with $\alpha < \beta,$ $M_{\alpha}(\chi) <\infty$ whenever $M_{\beta}(\chi) <\infty.$ 
\end{remark}

To establish the proposed results for the Kantorovich exponential sampling operator (\ref{kant}), we define the average type kernel as follows:
\begin{equation} \label{avg}
\displaystyle \bar{\chi}(t) = \int_{e^{\frac{-1}{2}}}^{e^{\frac{1}{2}}} \chi (tu) \frac{du}{u} = \int_{\frac{-1}{2}}^{\frac{1}{2}} \chi (te^p)dp,\ \ \ \ t \in \mathbb{R}^+.
\end{equation}

In the following lemma, we show that the average type kernel satisfies the conditions $(\chi_{1})$-$(\chi_{4}).$

\begin{lemma} \label{lemma1}
Let $\chi: \mathbb{R}^{+} \rightarrow \mathbb{R}$ be the kernel function satisfying $(\chi_{1})$ -$(\chi_{4})$ and $\bar{\chi}(t)$ be defined as in (\ref{avg}). Then $\bar{\chi}$ also satisfies $(\chi_{1})$ -$(\chi_{4}).$
\end{lemma}

\begin{proof}
Since $\chi$ is continuous, the averaged type kernel $\bar{\chi}$ is also continuous. Now for $u \in \mathbb{R}^+,$ we have 
$$ m_{0}(\bar{\chi}) =  \sum_{k=- \infty}^{\infty} \bar{\chi}(e^{-k}u)= \sum_{k=- \infty}^{\infty} \int_{\frac{-1}{2}}^{\frac{1}{2}} \chi (e^{-k} u e^p)dp = \displaystyle \int_{\frac{-1}{2}}^{\frac{1}{2}} dp = 1 .$$
Hence $\overline{\chi}$ satisfies $(\chi_{1}).$ Using the condition $(\chi_{2}),$ we get
\begin{eqnarray*}
m_{1} (\bar{\chi},u) &=&  \sum_{k=- \infty}^{\infty}\bar{\chi}(e^{-k}u)(k - \log u)\\
&=& \sum_{k=- \infty}^{\infty} \int_{\frac{-1}{2}}^{\frac{1}{2}} \chi (e^{-k} u e^p) (k - \log u + p - p) \ dp\\
&=& \int_{\frac{-1}{2}}^{\frac{1}{2}} \sum_{k=- \infty}^{\infty} \chi(e^{-k} u e^p)( k- \log(ue^p) + p) \ dp\\
&=&  m_1 (\chi , u) \int_{\frac{-1}{2}}^{\frac{1}{2}}dp  \ + \ \int_{\frac{-1}{2}}^{\frac{1}{2}}p \ dp = m_{1}^{\chi}.
\end{eqnarray*}
We define $\displaystyle M_{\beta} (\bar{\chi}):= \sup_{u \in \mathbb{R}^+} \sum_{k=- \infty}^{\infty} \vert \bar{\chi} (e^{-k} u) \vert  \ \vert k - \log u \vert ^{\beta}.$ For $\beta \geq 1,$  it is given that $M_{\beta} (\chi) < \infty.$ Now we show that $M_{\beta}(\bar{\chi}) < \infty.$ In view of (\ref{avg}), we have
\begin{eqnarray*}
M_{\beta} (\bar{\chi}) & \leq & \sum_{k=- \infty}^{\infty} \int_{\frac{-1}{2}}^{\frac{1}{2}}  \left( \vert \chi (e^{-k}u e^{p}) \vert \  \vert k - \log u + p -p \vert ^{\beta} \right) dp \\
& \leq & \int_{\frac{-1}{2}}^{\frac{1}{2}} \left( \sum_{k=- \infty}^{\infty} \vert \chi (e^{-k}u e^{p}) \vert \ \vert k - \log (u e^{p}) + p \vert ^{\beta} \right) dp \\
& \leq & \int_{\frac{-1}{2}}^{\frac{1}{2}} \sum_{k=- \infty}^{\infty} \vert \chi (e^{-k} u e^{p}) \vert \left( 2^{\beta - 1} \vert k - \log (u e^{p}) \vert ^{\beta} + \vert p \vert ^{\beta} \right) dp .
\end{eqnarray*}
Since $\beta \geq 1,$ then by using the inequality $\vert a+b \vert ^\beta \leq 2 ^{\beta-1} \left(\vert a \vert ^\beta + \vert b \vert^\beta \right),$ we obtain
\begin{eqnarray*}
M_{\beta} (\overline{\chi}) & \leq & 2 ^{\beta - 1} \int_{\frac{-1}{2}}^{\frac{1}{2}} \left( \sum_{k=- \infty}^{\infty} \vert \chi (e^{-k} u e^{p}) \vert \  \vert k - \log (u e^{p}) \vert ^{\beta} \right) dp \\ && + 2^{\beta - 1} \int_{\frac{-1}{2}}^{\frac{1}{2}} \left( \sum_{k=- \infty}^{\infty} \vert \chi (e^{-k} u e^{p}) \vert \  \vert  p \vert ^{\beta} \right) dp \\
& \leq & 2 ^{\beta - 1} \left( M_{\beta} (\chi) +  M_0 ({\chi}) \int_{\frac{-1}{2}}^{\frac{1}{2}} \vert  p \vert ^{\beta} dp  \right) \\
& \leq & 2 ^{\beta -1} M_{\beta} (\chi) + \frac{M_{0} (\chi)}{2^{\beta +1}(\beta +1)} ((-1)^{\beta +1}+1).
\end{eqnarray*}
Since $M_{\beta}(\chi) < \infty,$ so we have $M_{\beta}(\bar{\chi}) < \infty.$\\

Now we show that for $\gamma >0,$ 
$$\displaystyle\lim_{w\rightarrow \infty} \sum_{|w\log x-k|>w\gamma} |\bar{\chi}(e^{-k} x^w)| \ |w\log x-k|=0 $$ uniformly w.r.t. $ x\in \mathbb{R}^{+}.$ For $w>\dfrac{1}{\gamma},$ by using the definition of $\bar{\chi},$ we can write \\

\noindent 
$\displaystyle\sum_{|w\log x-k|>w\gamma} |\bar{\chi}(e^{-k} x^w)| \ |w\log x-k|$
\begin{eqnarray*}
&\leq &\sup_{p\in[-1/2,1/2]}\left(\sum_{|w\log x-k|>w\gamma} |\chi(e^{-k} x^we^{p})| \ |w\log x-k+\log(e^p)-\log(e^p)|\right)\\
&\leq &\sup_{p\in[-1/2,1/2]}\left(\sum_{|w\log x-k+\log(e^p)|>w\gamma-1/2} |\chi(e^{-k} x^we^{p})| \ (|w\log x-k+\log(e^p)|+|\log(e^p)|)\right)\\
&\leq &\sup_{y\in\mathbb{R}^{+}}\left(\sum_{|w\log y-k|>w\gamma-1/2} |\chi(e^{-k} y^w)| \left(|w\log y-k|+\frac{1}{2}\right)\right)\\
&\leq &\sup_{y\in\mathbb{R}^{+}}\left(\sum_{|w\log y-k|>w\gamma/2} |\chi(e^{-k} y^w)||w\log y-k|+\frac{1}{2}\sum_{|w\log y-k|>w\gamma/2} |\chi(e^{-k} y^w)|\right).
\end{eqnarray*}
Using the condition $(\chi_{4}),$ we deduce that
$$\displaystyle\lim_{w\rightarrow \infty} \sum_{|w\log x-k|>w\gamma} |\bar{\chi}(e^{-k} x^w)| \ |w\log x-k|=0 $$ uniformly w.r.t. $ x\in \mathbb{R}^{+}.$ This concludes the proof.
\end{proof}

Next we deduce the following result which will be useful to obtain a relation between the sampling series \eqref{genexp} and \eqref{kant}.

\begin{lemma}\label{lemma2}
Let $[a,b] \subset \mathbb{R}^{+}$ and $f:[a,b] \rightarrow \mathbb{R}$ be continuous. If 
\begin{equation} \label{antimellin}
F(x)= \int_{0}^{x} f(t) \frac{dt}{t},\ \ \ x \in [a,b].
\end{equation}
Then $F$ is Mellin differentiable and $ \ (\theta F)(x)=f(x),\ \forall x \in [a,b].$
\end{lemma}

\begin{proof} We have  
\begin{eqnarray*}
F(x)&=& \int_{a}^{x} f(t) \frac{dt}{t} = \int_{\log a}^{\log x} f(e^v) \ dv.
\end{eqnarray*}
This gives
\begin{eqnarray*}
F(s x)-F(x) &=& \int_{\log a}^{\log s x}f(e^v) \ dv - \int_{\log a}^{\log x}f(e^v) \ dv \\
&=&
\begin{cases}
     {\displaystyle \int_{\log x}^{\log s x} f(e^v) \ dv} &\quad\text{if} \ \ \ \ {s >1}\\
 {\displaystyle - \int_{\log x}^{\log s x} f(e^v) \ dv} &\quad\text{if} \ \ \ \ {s <1.}\\
\end{cases} \\
\end{eqnarray*}
On applying the mean value theorem for integral calculus, we get
$$ \int_{\log x}^{\log s x} f(e^v) \ dv = (\log s) f(e^{\xi}),$$ where $\xi \in [\log x, \log sx]$ for $s>1$ and $\xi \in [\log sx, \log x]$ for $s < 1.$ This gives $$F(sx)-F(x)=\log (s) f(e^{\xi}),\ \ \xi \in [\log x, \log sx].$$ This implies that
$$\frac{F(sx)-F(x)}{\log s} = f(s x) .$$ Taking limit as $s \rightarrow 1,$ we deduce
$$\lim_{s \rightarrow 1} \frac{F(sx)-F(x)}{\log s} = f(x).$$ This gives $(\theta F)(x)=f(x),\ \ x \in [a,b].$ Thus, the proof is completed. 
\end{proof}

\section{Main Results} \label{section3}

In this section, we derive the inverse approximation result and saturation order for the Kantorovich exponential sampling series $(I_{w}^{\chi}f).$ First we establish the relation between $(S_{w}^{\chi}f)$ and $(I_{w}^{\chi}f).$
Using the continuity of $\chi$ and Lemma \ref{lemma2}, we obtain
\begin{equation} \label{diff}
\theta \bar{\chi} (t) = \chi (te^{\frac{1}{2}}) - \chi (te^{\frac{-1}{2}}).
\end{equation}

\begin{lemma} \label{lemma3}
Let $f \in C(\mathbb{R}^{+})$ and $F$ be the Mellin anti-derivative of $f.$ Then for $x,w \in \mathbb{R}^{+},$ there holds
$$(I_w^\chi f)(x) = (\theta S^{\bar{\chi}}_w F) (xe^{\frac{1}{2w}}).$$
\end{lemma}

\begin{proof} Using (\ref{antimellin}), we can write
\begin{eqnarray*}
(I_w^{\chi} f)(x) &=& \displaystyle \sum_{k=-\infty}^{\infty} \chi (e^{-k} x^w)\  w \int_{\frac{k}{w}}^{\frac{k+1}{w}} f (e^u)\ du \\
&=& \sum_{k=-\infty}^{\infty} \chi (e^{-k} x^w) \ w \left( F (e^{\frac{k+1}{w}}) - F (e^{\frac{k}{w}}) \right)\\
&=& w \left( \sum_{k=-\infty}^{\infty} \chi (e^{-k} x^w)F (e^{\frac{k+1}{w}}) -  \sum_{k=-\infty}^{\infty} \chi (e^{-k} x^w)F (e^{\frac{k}{w}}) \right).
\end{eqnarray*}
Setting $\widetilde{k} = k+1$ in the first term of the above expression, we have
\begin{eqnarray} \label{rel} \nonumber
(I_w^{\chi} f)(x) &=& w \left( \sum_{\widetilde{k}=-\infty}^{\infty} \chi (e^{\widetilde{-k}} x^w)F (e^{\frac{\widetilde{k}}{w}}) -  \sum_{k =-\infty}^{\infty} \chi (e^{-k} x^w) F (e^{\frac{k}{w}}) \right)\\  \nonumber
& =&  w \left( \sum_{\widetilde{k}=-\infty}^{\infty} \chi (e^{\widetilde{-k}} x^w e^{\frac{1}{2}} e^{\frac{1}{2}}) F (e^{\frac{\widetilde{k}}{w}}) \right) - w \left( \sum_{k=-\infty}^{\infty} \chi (e^{-k}  x^w e^{\frac{1}{2}} e^{\frac{-1}{2}}) F (e^{\frac{k}{w}}) \right)\\ \nonumber
& =& \sum_{k=-\infty}^{\infty} F (e^{\frac{k}{w}})\ w \left(\chi (e^{-k} x^w e^{\frac{1}{2}} e^{\frac{1}{2}}) - \chi (e^{-k} x^w e^{\frac{1}{2}} e^{\frac{-1}{2}}) \right). \\ \nonumber
\end{eqnarray}
From (\ref{diff}), we get 
$$ \displaystyle (I_w^{\chi} f)(x)= \sum_{k=-\infty}^{\infty} F (e^{\frac{k}{w}}) \ w \ (\theta \bar{\chi}) (e^{-k} x^w e^{\frac{1}{2}}).$$
Since $(\theta f)(x)=x f^{'}(x),$ thus we obtain $$(\theta S_{w}^{\bar{\chi}}F)(xe^{\frac{1}{2w}}) =\sum_{k=-\infty}^{\infty} F (e^{\frac{k}{w}}) \ w \ (\theta \bar{\chi})(e^{-k} x^w e^{\frac{1}{2}}).$$ Hence, the proof is established.
\end{proof}

Next we establish the asymptotic formula for the  series $(S_{w}^{\chi}f).$ This result is required to derive the saturation order for  the Kantorovich exponential sampling series $(I_{w}^{\chi}f).$

\begin{thm}\label{xyz1}
If $f: \mathbb{R}^{+} \rightarrow \mathbb{R}$ be
differentiable such that $(\theta f)$ is log-uniformly continuous and bounded on $\mathbb{R}^{+},$ then
$$\displaystyle \lim_{w \to + \infty} w\ ((S_{w}^{\chi}f)(x)-f(x))=m_{1}^{\chi} \ (\theta f)(x),$$
uniformly on $\mathbb{R}^{+}.$
\end{thm}

\begin{proof}
Setting $u=k/w$ in the first order Mellin Taylor formula, we obtain
$$f(e^{k/w})=f(x)+ \left( \frac{k}{w} - \log x \right) (\theta f)(\xi), \ \ \ \ \xi \in (k/w,\log x).$$ Operating $\displaystyle \sum_{k=-\infty}^{\infty}\chi(e^{-k} x^{w})$ on both sides of the above expression, we get
$$ (S_{w}^{\chi}f)(x)= f(x)+ \sum_{k=-\infty}^{\infty} \chi(e^{-k}x^{w})(\theta f )(\xi) \left(\frac{k}{w}-\log x \right) := f(x)+R_{w}(x).$$
The above expression implies 
$$w \left( (S_{w}^{\chi}f)(x)-f(x) \right) - m_{1}^{\chi} (\theta f)(x) = w \ R_{w}(x) - m_{1}^{\chi} (\theta f)(x).$$
For $\delta >0,$ we can write\\

\noindent 
$w \ R_{w}(x) - m_{1}^{\chi} \ (\theta f)(x)$
\begin{eqnarray*}
 &=& \sum_{k=-\infty}^{\infty} \chi (e^{-k}x^{w})(k-w \log x)((\theta f)(\xi)-(\theta f)(x)) \\
&=& \left(\sum_{\left|\frac{k}{w}-\log x \right|<\delta}+\sum_{\left|\frac{k}{w}-\log x \right|\geq\delta}\right)\chi(e^{-k}x^{w})(k-w\log x)((\theta f)(\xi)-(\theta f)(x))\\
&:=& I_{1}+I_{2}.
\end{eqnarray*}
Since $(\theta f)$ is log-uniformly continuous, $\forall \epsilon >0,$ $\exists \,\ \delta >0$ such that $\left| (\theta f)(\xi)- (\theta f)(x) \right| < \epsilon,$ for $\left|\dfrac{k}{w} - \log x \right| < \delta.$ This gives
\begin{eqnarray*}
|I_{1}|&\leq &  \epsilon \sum_{\left|\frac{k}{w}-\log x \right|<\delta}| \chi(e^{-k}x^{w})| |k-w\log x| \\
&\leq& \epsilon \ M_{1}(\chi).
\end{eqnarray*}
Since $\epsilon>0$ is arbitrary, we obtain $I_1\rightarrow 0$ as $w\rightarrow\infty.$ In view of boundedness of $(\theta f),$ we get  
\begin{eqnarray*}
| I_{2} | &\leq & 2 \| \theta f \|_{\infty}\sum_{|{k}-w\log x| \geq w \delta}| \chi(e^{-k}x^{w})| |k-w\log x|.
\end{eqnarray*}
Using the condition $(\chi_{4}),$ we obtain $I_2\rightarrow 0$ as $w\rightarrow\infty.$ On collecting the estimates $I_{1}-I_{2},$ we obtain 
$$ \lim_{w \to +\infty} w \left[ (S_{w}^{\chi}f)(x)-f(x) \right] =m_{1}^{\chi} \ (\theta f)(x).$$
Thus, the proof is completed. 
\end{proof}

This gives the following corollary.
\begin{corollary}\label{xyz2}
Let $f:\mathbb{R}^{+}\to \mathbb{R}$ be such that $(\theta f) \in \mathcal{C}(\mathbb{R}^+).$ Then for any $x \in \mathbb{R}^+$ and $w>0,$ we have
$$ \lim_{w \to +\infty} w[(S_{w}^{\chi}f)(x e^{1/2w})-f(x)]=(2 m_{1}^{\chi} +1) \frac{(\theta f)(x)}{2},$$
uniformly on $\mathbb{R}^{+}.$
\end{corollary}

\begin{proof} Using Theorem \ref{xyz1}, we write 
$$\Big| w \left[(S_{w}^{\chi}f)(u)-f(u) \right] - m_{1}^{\chi} \ (\theta f)(u) \Big| < \epsilon \ ,$$ hold for any $u \in \mathbb{R}^{+}$ and for every fixed $\epsilon>0.$ If we take $u=x  e^{1/2w},$ then from the above estimate, we obtain
$$\Big| w \left[(S_{w}^{\chi}f)(x e^{1/2w})-f(x e^{1/2w}) \right]- m_{1}^{\chi} \ (\theta f)(x e^{1/2w})\Big| < \epsilon.$$
Thus, we have \\

\noindent 
$ \Big| w \left[(S_{w}^{\chi}f)(x e^{1/2w})-f(x) \right ]-(2 m_{1}^{\chi} +1) \dfrac{(\theta f)(x)}{2} \Big| $
\begin{eqnarray*}
& < & \epsilon + \Bigg| w(f(x e^{1/2w})-f(x))-\frac{(\theta f)(x)}{2}\Bigg| + \Big| m_{1}^{\chi} \left( (\theta f)(x)-(\theta f)(x e^{\frac{1}{2w}}) \right) \Big| \\
&<& \epsilon+ \frac{1}{2}\Bigg| \frac{f(x e^{1/2w})-f(x)}{1/2w}-\frac{(\theta f)(x)}{2}\Bigg| + \Big| m_{1}^{\chi} \left( (\theta f)(x) - (\theta f)(x e^{\frac{1}{2w}}) \right) \Big| \\
&:=& \epsilon + I_{3}+I_{4}.
\end{eqnarray*}
Since $\displaystyle \lim_{w \rightarrow \infty} 2w \left( f(x e^{1/2w})-f(x) \right) = \frac{(\theta f)(x)}{2},$ we deduce that $I_{3} \rightarrow 0$ as $w \rightarrow \infty.$ Since $\theta f$ is log-uniformly continuous, we have $$|\theta f(x)-\theta f(x e^{1/2w})| < \epsilon,$$ whenever $|\log x - \log (x e^{1/2w})|< \delta.$
Therefore, we obtain 
$$ \Bigg| w \left[(S_{w}^{\chi}f)(x e^{1/2w})-f(x) \right] - 2 ( m_{1}^{\chi} +1) \frac{(\theta f)(x)}{2} \Bigg| < \epsilon .$$ Hence we get the desired result. 
\end{proof}

In what follows, we shall define the class of log-H\"{o}lderian functions as
$$ L_{\alpha} := \{ f: I \rightarrow \mathbb{R} \ : \exists \ \mbox{K $>$ 0 \ s.t.} \ \ | f(x)-f(y)| \leq  K | \log x - \log y |^{\alpha};  \hspace{0.11cm}x,y \in I \},$$ with $I \subseteq \mathbb{R}^{+}$ and $0 < \alpha \leq 1.$ Now we prove the following direct approximation result.

\begin{thm}\label{xyz4} 
Let $\chi$ be a kernel function and $f \in L_\alpha.$ Then the following holds 
$$ \| I_{w}^{\chi}f - f\|_\infty = \mathcal{O}(w^{-\alpha}) ~~~ \mbox{as} \ \ w \rightarrow \infty.$$
\end{thm}

\begin{proof} Consider 
\noindent $ |(I_{w}^{\chi}f)(x) - f(x)|$
\begin{eqnarray*}
&=& \Big|\sum_{k= -\infty}^{\infty} \chi(e^{-k}x^{w}) w \int_{k/w}^{{k+1}/w} [f(e^u) - f(x)]\ du \Big| \\
 &\leq & \sum_{k= -\infty}^{\infty} | \chi(e^{-k}x^{w})| w \int_{k/w}^{{k+1}/w} |f(e^u) - f(x)|\ du.
\end{eqnarray*}
Since $f \in L_\alpha,$ so we obtain 
 \begin{eqnarray*}
|(I_{w}^{\chi}f)(x) - f(x)|&\leq & \sum_{k= -\infty}^{\infty} | \chi(e^{-k}x^{w})| w \int_{k/w}^{{k+1}/w} |u - \log x|^\alpha du\\
& \leq & \frac{w}{(\alpha + 1)} \sum_{k= -\infty}^{\infty} | \chi(e^{-k}x^{w})| \left[ \ \Big| \frac{k+1}{w} - \log x \Big|^{\alpha +1} - \ \Big| \frac{k}{w} - \log x \Big|^{\alpha +1} \right].
\end{eqnarray*}
Since $|a+b|^{\alpha+1} \leq 2^{\alpha} \left( |a|^{\alpha} + |b|^{\alpha} \right)$ for $\alpha >0,$ we can write\\

\noindent $ |(I_{w}^{\chi}f)(x) - f(x)|$
\begin{eqnarray*}
& \leq & \frac{w}{(\alpha + 1)} \sum_{k= -\infty}^{\infty} | \chi(e^{-k}x^{w})| \left[ \ 2^{\alpha} \left( \Big| \frac{k}{w} - \log x \Big|^{\alpha +1} + \frac{1}{w^{\alpha+1}} \right) - \ \Big| \frac{k}{w} - \log x \Big|^{\alpha +1} \right]\\
& \leq & \frac{w}{(\alpha + 1)} \sum_{k= -\infty}^{\infty} | \chi(e^{-k}x^{w})| \ \left[ (2^{\alpha}-1)\Big| \frac{k}{w} - \log x \Big|^{\alpha +1}  + \frac{1}{w^{\alpha+1}} \right]\\
& \leq & \frac{w^{-\alpha}}{(\alpha + 1)} \left( (2^{\alpha} -1) M_{\alpha +1}(\chi) + M_{0}(\chi) \right).
\end{eqnarray*}
This completes the proof.
\end{proof}

We derive the saturation order for the Kantorovich exponential sampling series as follows.

\begin{thm} \label{xyz3}
Let $\chi$ be a kernel function such that $m_{1}^{\chi} \neq -1/2$ and let $f \in \mathcal{C}(\mathbb{R}^+)$ be such that 
$$ \|I_{w}^\chi f - f \|_\infty = o(w^{-1})  ~~~~ as~~~ w \to + \infty. $$ 
Then $f$ is constant on $\mathbb{R}^+.$
\end{thm}

\begin{proof}
Let $\phi \in C_{c}^\infty(\mathbb{R}^+) $ be fixed. We define 
\begin{equation}\label{1}
G_f(\phi): = w\int_{\mathbb{R}^+} \big{[(I_{w}^\chi f)-f(x)\big]}\phi(x) \frac{dx}{x}, \ \ \ \ w>0.
\end{equation}
We assume that $[a,b] \subset\mathbb{R}^+$ be such that the compact support of $\phi$ is properly contained in $[a,b],$ i.e. $supp(\phi) \subset[a,b].$ This gives $\phi(a) =0= \phi(b).$ So, the integral \eqref{1} reduces to 
\begin{equation*}
G_f(\phi) = w\int_{a}^{b} \big{[(I_{w}^\chi f)-f(x)\big]}\phi(x)\frac{dx}{x}, \ \ \ \ w>0.
\end{equation*}
Using Lemma \ref{lemma3}, we obtain
\begin{align*}
	 G_f(\phi) &= w\int_{a}^{b} \big{[(\theta S_{w}^{\bar{\chi}}F)(xe^{1/2w})-f(x)\big]}\phi
	(x)\frac{dx}{x} \\
	&= w\int_{a}^{b} \big{[(\theta S_{w}^{\bar{\chi}}F)(xe^{1/2w})-(\theta F)(x)\big]}\phi
	(x)\frac{dx}{x},
\end{align*} 
where $F$ is the Mellin anti-derivative of $f,$ i.e. 
$\displaystyle F(x) = \int_{0}^{x}f(t)\frac{dt}{t},~~x \in \mathbb{R}^+.$ Using integration by parts in the Mellin-sense, we obtain 
\begin{align*}
G_f(\phi) = w \left\{\phi(x) \int \big{[(\theta S_{w}^{\bar{\chi}}F)(xe^{1/2w})-(\theta F)(x)\big]}\frac{dx}{x} \right \}_{a}^{b}
	 \\- w\int_{a}^{b} \big{[(S_{w}^{\bar{\chi}}F)(xe^{1/2w})- F(x)\big]}(\theta \phi)(x)\frac{dx}{x}.
\end{align*}
Since $\phi(a)= 0= \phi(b),$ we have
\begin{equation*}
	G_f(\phi) =	- w\int_{a}^{b} \big{[(S_{w}^{\bar{\chi}}F)(xe^{1/2w})- F(x)\big]}(\theta \phi)(x)\ \frac{dx}{x}.
\end{equation*}
Applying the limit $w \rightarrow \infty$ on both sides of the above equation and using Vitali's convergence theorem, we get
\begin{equation}\label{2}
\lim\limits_{w\to \infty}G_f(\phi) = -(m_{1}^{\chi} + 1/2) \int_{a}^{b} (\theta \phi)(x) (\theta F)(x)\ \frac{dx}{x}.
\end{equation}
Since $(I_{w}^{\chi}f)$ converges to $f$ uniformly as $w \to \infty,$ we obtain
$$ 0=-(m_{1}^{\chi} + 1/2) \int_{a}^{b} (\theta \phi)(x)(\theta F)(x)\frac{dx}{x}.$$
As $(\theta F)(x)=f(x),$ we have
$$ 0= -(m_{1}^{\chi} + 1/2) \int_{a}^{b} (\theta \phi)(x) f(x) \frac{dx}{x},$$
Since $\phi \in C_{c}^{\infty}(\mathbb{R}^{+})$ is arbitrary function, $f$ is constant in $\mathbb{R}^{+}.$
Hence, the highest order of convergence that $(I_{w}^\chi f)$ can achieve for $f \in \mathcal{C}(\mathbb{R}^+)$ is one, provided $m_1(\chi,u) \neq -1/2.$ Hence, the result is proved. 
\end{proof}

For $f \in \mathcal{C}(\mathbb{R}^+),$ the logarithmic modulus of continuity is defined as
$$ \omega(f,\delta):= \sup \{|f(u)-f(v)|: \  |\log u-\log v| \leq \delta,\  \ \delta \in \mathbb{R}^{+}\} .$$
For every $\delta > 0$ and $u,v \in \mathbb{R}^+,$ \ $\omega$ \ has the following properties:
\begin{itemize}
\item[(a)] $\omega(f, \delta) \rightarrow 0$ as $\delta \rightarrow 0.$

\item[(b)] $\displaystyle|f(u) - f(v)| \leq \omega(f,\delta) \left( 1+ \frac{|\log u - \log v|}{\delta} \right).$
\end{itemize}
For more details on modulus of continuity, we refer to \cite{mamedeo,bardaro9}. \\

Now we prove the proposed inverse approximation result for the Kantorovich exponential sampling series $I_{w}^{\chi}.$

\begin{thm} \label{xyz5} 
Let $\chi$ be differentiable and $M_{1}(\theta \chi) < +\infty$ and $f \in \mathcal{C}(\mathbb{R}^+)$ be such that 
$$ \|I_{w}^{\chi}f - f\|_\infty = \mathcal{O}(w^{-\alpha}) ~~~as~~~ w \to \infty $$
with $0< \alpha \leq 1.$ Then $f$ belongs to $L_\alpha.$
\end{thm}

\begin{proof} 
Since $\|I_{w}^{\chi}f - f\|_\infty = \mathcal{O}(w^{-\alpha}) ~~~as~~~ w \to \infty,$ there exist $K$ and $w_0$ such that 
$$\|I_{w}^{\chi}f - f\|_\infty \leq K w^{-\alpha}, \ \ \ \mbox{for every}\ \ w>w_0.$$
Now for every fixed $x,y\in\mathbb{R}^+$ and for $x \neq y,$ we can write 
\begin{align*}
	|f(x)-f(y)| &= |f(x)-(I_{w}^{\chi}f)(x) + (I_{w}^{\chi}f)(x)-f(y)+(I_{w}^{\chi}f)(y)-(I_{w}^{\chi}f)(y)| \\
	&\leq |f(y)-(I_{w}^{\chi}f)(y)| + |f(x)-(I_{w}^{\chi}f)(x)| + |(I_{w}^{\chi}f)(x)-(I_{w}^{\chi}f)(y)| \\
	&\leq Kw^{-\alpha} + |(I_{w}^{\chi}f)(x) - (I_{w}^{\chi}f)(y)| +Kw^{-\alpha} \\
	&\leq 2Kw^{-\alpha} + |(I_{w}^{\chi}f)(x)- (I_{w}^{\chi}f)(y)| \\
	&\leq 2Kw^{-\alpha} + \Bigg| \int_{y}^{x}(\theta I_{w}^{\chi}f)(t) \frac{dt}{t}\Bigg|  =: I_{1}+I_{2}.
\end{align*}
Now we estimate $I_2.$ It is easy to see that
\begin{equation}\label{4}
(\theta I_{w}^{\chi}f)(x) = \sum_{k=-\infty}^{\infty} (\theta {\chi})(e^{-k}x^w)\ w^{2} \int_{k/w}^{{k+1}/w} f(e^u) \ du.
\end{equation}
Let $\textbf{1}$ represents the constant function such that $\textbf{1}(x)=1,\ \forall x \in \mathbb{R}^{+}.$ We can easily observe that $(I_{w}^{\chi} \textbf{1})(x)= 1,\ \forall x \in \mathbb{R}^{+}.$ This gives 
$$ (\theta I_{w}^{\chi} \textbf{1})(x)= w \sum_{k=-\infty}^{\infty} \chi^{'} (e^{-k}x^w) \ (e^{-k}x^w)=0.$$
Again we have
\begin{eqnarray*}
|(\theta I_{w}^{\chi}f)(x)| &=& \big|(\theta I_{w}^{\chi}f)(x) - f(x) \ (\theta I_{w}^{\chi} \textbf{1})(x) \big| \\
&=& \Big| w \sum_{k= -\infty}^{\infty} (\theta \chi)(e^{-k}x^w)\ w \int_{k/w}^{{k+1}/w} f(e^u) \ du - w f(x) \sum_{k= -\infty}^{\infty} (\theta \chi)(e^{-k}x^w) \big|\\
&=& w^{2} \Big| \sum_{k= -\infty}^{\infty} (\theta \chi)(e^{-k}x^w) \int_{k/w}^{{k+1}/w} (f(e^u)-f(x)) \ du \Big| \\
& \leq & w^{2} \sum_{k= -\infty}^{\infty} |(\theta \chi)(e^{-k}x^w)| \int_{k/w}^{{k+1}/w} |f(e^u)-f(x)| \ du.
\end{eqnarray*}
Using property (b) of $\omega$ and choosing $\delta=1/w$ to get
\begin{eqnarray*}
|(\theta I_{w}^{\chi}f)(x)| & \leq & w^{2} \sum_{k= -\infty}^{\infty} |(\theta \chi)(e^{-k}x^w)| \int_{k/w}^{{k+1}/w} \omega(f,1/w) \ (1+ w \ |u-\log x|)\ du \\
& \leq & w^{3} \ \omega(f,1/w)\  \sum_{k= -\infty}^{\infty} |(\theta \chi)(e^{-k}x^w)| \int_{k/w}^{{k+1}/w} |u- \log x|\ du \\&&
+w \ \omega(f,1/w)\  \sum_{k= -\infty}^{\infty} |(\theta \chi)(e^{-k}x^w)| \\
& \leq & w \ \omega(f,1/w) \ M_{0}(\theta \chi)+ \frac{w}{2} \ \omega(f,1/w) \ (M_{0}(\theta \chi)+2 M_{1}(\theta \chi))\\
& \leq & \frac{w}{2} \ \omega(f,1/w) \left( 3 M_{0}(\theta \chi)+ 2 M_{1}(\theta \chi) \right).
\end{eqnarray*}
This gives
\begin{eqnarray*}
|f(x)-f(y)| & \leq & 2K w^{-\alpha} + \frac{w}{2} \ \omega(f,1/w)\int_{y}^{x} \left( 3 M_{0}(\theta \chi)+ 2 M_{1}(\theta \chi) \right) \frac{dt}{t} \\
& \leq & 2K w^{-\alpha} + \frac{w}{2} \ \omega(f,1/w) \left( 3 M_{0}(\theta \chi)+ 2 M_{1}(\theta \chi) \right) |\log x - \log y|.
\end{eqnarray*}
Since $M_{0}(\theta \chi)$ and $M_{1}(\theta \chi)$ are finite, we define $N:= \mbox{max} \{M_{0}(\theta \chi),M_{1}(\theta \chi) \}$ to obtain
$$ \omega(f,\delta) \leq 2K w^{-\alpha}+ w \delta \ \omega(f,1/w) \left( \frac{5N}{2} \right).$$
For any fixed $\delta >0,$ there exists sufficiently large $w$ such that $\frac{1}{w} < \delta.$ Now, in view of monotonicity of $\omega,$ we write
$$ \omega(f,1/w) < \omega (f,\delta).$$
This gives $\omega(f,\delta) \leq H \ w^{-\alpha} ,$ where $ \displaystyle H := \frac{4K}{5-2N}.$ Hence, we conclude that $f \in L_{\alpha},\ 0<\alpha \leq 1.$ Thus, we obtain the desired result. 
\end{proof}

\section{Examples of kernel} \label{section4}

In this section, we discuss some examples of kernel functions satisfying the assumptions of our theory.
\subsection{Mellin B-spline kernel}
 We start with the Mellin B-spline kernel of order $n \in \mathbb{N}$ which are defined by (see \cite{bardaro7})
$$\overline{{B}_{n}}(t):= \frac{1}{(n-1)!} \sum_{j=0}^{n} (-1)^{j} {n \choose j} \bigg( \frac{n}{2}+\log t-j \bigg)_{+}^{n-1},\,\,\,\ t\in \mathbb{R}^{+},$$ where $(x)_{+} := \max \{x,0\},$ $ x \in \mathbb{R}.$ We observe that for $t\in \mathbb{R}^{+},$ $B_n(\log t)=\overline{{B}_{n}}(t),$ where 
$${B}_{n}(x):= \frac{1}{(n-1)!} \sum_{j=0}^{n} (-1)^{j} {n \choose j} \bigg( \frac{n}{2}+x-j \bigg)_{+}^{n-1},\,\,\,\ x\in \mathbb{R},$$ denotes the classical $B$-splines of order $n \in \mathbb{N}.$ Also, we noted $B_n(x)=\overline{{B}_{n}}(e^x),$ $x\in \mathbb{R}.$
Now using the definition of the Mellin transform, we have 
$$[\overline{{B}_{n}}]^{\wedge}_M (is)= \int_0^{\infty} \overline{{B}_{n}}(t)\ t^{is}\ \frac{dt}{t}, \ \ \ s\in\mathbb{R}.$$ Subtituting $u=\log t,$ we easily obtain 
\begin{eqnarray*}
[\overline{{B}_{n}}]^{\wedge}_M (is)&=&\int_{-\infty}^{\infty} \overline{{B}_{n}}(e^u)\ e^{isu}\ du
=\int_{-\infty}^{\infty} {B}_{n}(u)\ e^{isu}\ du
= \widehat{B_n}(-s),
\end{eqnarray*}
where $\widehat{f}(u) := \displaystyle\int_{-\infty}^{\infty} f(x) e^{-iux}dx, u \in \mathbb{R}$ denotes the Fourier transform of the function $f.$ Now, let $f(x):=\displaystyle\sum_{k=- \infty}^{\infty}\overline{{B}_{n}}(e^{k}x).$ This $f$ is a recurrent function with fundamental interval $[1,e].$ That is $f(ex)=f(x),$ $\forall x\in\mathbb{R}^{+}.$ The Mellin-Fourier cofficient $m_{k}(f)$ of $f$ is given by 
\begin{eqnarray*}
m_{k}(f):=\displaystyle\int_{1}^{e}f(x)x^{2k\pi i}\frac{dx}{x}
=\displaystyle\int_{1}^{e}\sum_{j=- \infty}^{\infty}\overline{{B}_{n}}(e^{j}x)x^{2k\pi i}\frac{dx}{x}
=\sum_{j=- \infty}^{\infty}\displaystyle\int_{1}^{e}\overline{{B}_{n}}(e^{j}x)x^{2k\pi i}\frac{dx}{x}.
\end{eqnarray*}
Substituting $u=e^{j}x,$ we obtain 
\begin{eqnarray*}
m_{k}(f)=\sum_{j=- \infty}^{\infty}\displaystyle\int_{e^j}^{e^{j+1}}\overline{{B}_{n}}(u){\left(\frac{u}{e^j}\right)}^{2k\pi i}\frac{du}{u}
=\displaystyle\int_{0}^{\infty}\overline{{B}_{n}}(u){u}^{2k\pi i}\frac{du}{u}=[\overline{{B}_{n}}]^{\wedge}_M (2k\pi i).
\end{eqnarray*}
Therefore, we get the following Mellin Poisson summation formula
\begin{eqnarray} \label{summation}
\displaystyle\sum_{k=- \infty}^{\infty}\overline{{B}_{n}}(e^{k}x)=\sum_{k=- \infty}^{\infty}[\overline{{B}_{n}}]^{\wedge}_M (2k\pi i)x^{-2k\pi i}. 
\end{eqnarray}
It is easy to see that 
\begin{eqnarray} \label{splinefourier}
[\overline{{B}_{n}}]^{\wedge}_M (c+is) = \bigg( \frac{\sin(\frac{s}{2})}{(\frac{s}{2})} \Bigg)^{n},\ \ \hspace{0.5cm} s\neq 0, c=0.
\end{eqnarray}
Therefore, we have 
\begin{equation} \label{e4}
[\overline{{B}_{n}}]^{\wedge}_M (2k\pi i)=\widehat{B_n}(-2k\pi i)=
     \begin{cases}
     {1,} &\quad\text{if} \ \  {k=0} \\
     {0,} &\quad \ \  {\text{otherwise.}}\\
   \end{cases}
\end{equation}
Using the Mellin Poisson summation formula (\ref{summation}), we get 
$$\displaystyle\sum_{k=- \infty}^{\infty}\overline{{B}_{n}}(e^{k}x)=1, \ \ \forall x\in\mathbb{R}^{+}.$$ 
Hence $ \overline{{B}_{n}}$ satisfies the condition $(\chi_{1}).$\\

Now we establish the condition $(\chi_{2}).$ Again using the Mellin transform and by differentiating under the sign of the integral, we get 
$$\frac{d}{ds}\left([\overline{{B}_{n}}]^{\wedge}_M (is)\right)=\frac{d}{ds}\left(\displaystyle\int_{0}^{\infty}\overline{{B}_{n}}(t){t}^{is}\frac{dt}{t}\right)=\int_{0}^{\infty}\overline{{B}_{n}}(t){t}^{is}(i\log t)\frac{dt}{t}.$$
Similarly, we can easily obtain 
$$\frac{d^{j}}{ds^{j}}\left([\overline{{B}_{n}}]^{\wedge}_M (is)\right)=\int_{0}^{\infty}\overline{{B}_{n}}(t){t}^{is}(i\log t)^{j}\frac{dt}{t}=(i)^{j}[\overline{f_j}]^{\wedge}_M (is),$$
where $f_{j}(t)=\overline{{B}_{n}}(t)(\log t)^{j}.$ 
Thus, we get 
\begin{eqnarray}\label{Poisson}
\displaystyle\sum_{k=- \infty}^{\infty}\overline{{B}_{n}}(e^{-k}x)(k-\log x)^j=\sum_{k=- \infty}^{\infty}(-i)^{j}\frac{d^{j}}{ds^{j}}\overline{{B}_{n}}]^{\wedge}_M (2k\pi i)x^{-2k\pi i}.
\end{eqnarray}
Again from (\ref{splinefourier}), we obtain
\begin{eqnarray*}
\dfrac{d}{ds} \left( [\overline{{B}_{n}}]^{\wedge}_M (is)\right)= n \bigg( \frac{\sin(\frac{s}{2})}{(\frac{s}{2})} \Bigg)^{n-1} \left( \dfrac{s \cos(s/2)-2 \sin(s/2)}{s^{2}}\right),\ \ \hspace{0.5cm} s \neq 0,
\end{eqnarray*}
which in-turn gives 
$$\displaystyle \dfrac{d}{ds} \left( [\overline{{B}_{n}}]^{\wedge}_M (is)\right) (2k \pi i)=0,\ \ k \in \mathbb{Z}.$$ Therefore, we obtain 
$m_{1}(\overline{{B}_{n}},x)=0,\ \forall n \in \mathbb{N},$ by using (\ref{Poisson}).  Hence the condition $(\chi_{2})$ is also satisfied.\\

As $\overline{{B}_{n}}$ is compactly supported, there exists $\nu>0$ such that $supp(\overline{{B}_{n}})\subseteq [e^{-\nu}, e^{\nu}].$ Thus, we get 
$|\{k:e^{-\nu}\leq e^{-k}u \leq e^{\nu}\}|\leq 2 [\nu]+1,$ $\forall u\in\mathbb{R}^+,$ where $[.]$ denotes the integer part. Hence we obtain 
\begin{equation*}
\sum_{k=-\infty}^{\infty} |\overline{{B}_{n}}(e^{-k} u)| \ |k- \log u|^{\beta} \leq (2 [\nu]+1)\left(\sup_{u\in\mathbb{R}^+}\overline{{B}_{n}}(u)\right)\nu^{\beta}<\infty.
\end{equation*}
This establishes the condition $(\chi_{3}).$\\

Again using the fact that 
$supp(\overline{{B}_{n}})\subseteq [e^{-\nu}, e^{\nu}],$ for some $\nu>0.$ Choosing $w\gamma>\nu,$ we get 
$\displaystyle \sum_{|k-w\log x|>w\gamma} |\overline{{B}_{n}}(e^{-k} x^w)| \ |k- w\log x|=0,$ $\forall x\in\mathbb{R}^+.$ Therefore, the condition $(\chi_{4})$ is satisfied. To verify the conditions which are used in the hypothesis of the theorem, we consider the third order Mellin $B$-spline kernel by 
\begin{equation*}
\overline{{B}_{3}}(x) =
     \begin{cases}
      {-\frac{1}{2} \left(\frac{3}{2}+ \log x \right)^{2},} &\quad\text{} \ \ { \text{$e^{-3/2} < x < e^{-1/2},$}}\\
      {\frac{3}{4} -  \log^{2} x ,} &\quad\text{} \ \ {e^{-1/2} < x < e^{1/2},}\\
 {-\frac{1}{2} \left(\frac{3}{2}- \log x \right)^{2},} &\quad\text{} \ \  { \text{$e^{1/2} < x < e^{3/2},$}}\\
  {0,} &\quad\text{} \ \  { \text{otherwise}.}\\
\end{cases}
\end{equation*}

The Mellin derivative of $\overline{{B}_{3}}(x)$ is given by 
\begin{equation*}
(\theta \overline{{B}_{3}})(x) =
     \begin{cases}
      {-\left(\frac{3}{2}+ \log x \right),} &\quad\text{} \ \  { \text{$e^{-3/2} < x < e^{-1/2},$}}\\
      { - 2 \log x ,} &\quad\text{} \ \  {e^{-1/2} < x < e^{1/2},}\\
 {\frac{3}{2}- \log x ,} &\quad\text{} \ \  { \text{$e^{1/2} < x < e^{3/2},$}}\\
  {0,} &\quad\text{} \ \ \  { \text{otherwise}.}\\
\end{cases}
\end{equation*}
Evidently, $supp(\theta \overline{{B}_{3}})\subseteq \left(e^{-3/2},e^{3/2} \right).$ Hence, $M_\beta(\theta \overline{{B}_{3}})$ are finite. So it satisfies the assumption of Theorem \ref{xyz5}. Moreover, we have $m_{1}^{\overline{{B}_{3}}} := m_{1}(\overline{{B}_{3}},x) = 0,\ \forall x \in \mathbb{R}^{+}.$ Thus, $\overline{{B}_{3}}(x)$ also satisfies the assumption of Theorem \ref{xyz3}.
\subsection{Mellin Jackson kernel}

Next we consider Mellin Jackson kernel. For $c \in \mathbb{R}, \alpha \geq 1,n \in \mathbb{N}$ and $x \in \mathbb{R}^{+},$ the Mellin Jackson kernel is defined by (see \cite{bardaro7})
$$ \overline{J_{\alpha,n}}(x):= C_{\alpha,n}\ x^{-c} \textit{sinc}^{2n} \left(\frac{\log x}{2 \alpha n \pi} \right),$$
where $\displaystyle C^{-1}_{\alpha,n} := \int_{0}^{\infty} \textit{sinc}^{2n} \left(\frac{\log x}{2 \alpha n \pi} \right)\frac{dx}{x}$ and
$
sinc(u) =
     \begin{cases}
      {\dfrac{\sin \pi u}{\pi u},} &\quad\text{} \ \  { u \neq 0}\\
      { 1,} &\quad\text{} \ \  {u=0}.\\
 \end{cases}
$

It is evident that 
$  \overline{J_{\alpha,n}}(x) = x^{-c} {J_{\alpha,n}}(\log x),$ where ${J_{\alpha,n}}$ represents the generalized Jackson kernel (see \cite{k2007}) given by
$$ {J_{\alpha,n}}(x):= c_{\alpha,n} \textit{sinc}^{2n} \left(\frac{x}{2 \alpha n \pi} \right),$$
where $\displaystyle c^{-1}_{\alpha,n} := \int_{-\infty}^{\infty} \textit{sinc}^{2n} \left(\frac{x}{2 \alpha n \pi} \right)dx \ $ is constant. Now we have 
\begin{eqnarray*}
\| \overline{J_{\alpha,n}} \|_{X_c} &=& \int_0^{\infty} | \overline{J_{\alpha,n}}(x)| x^{c}\  \frac{dx}{x}=\int_0^{\infty} |{J_{\alpha,n}}(\log x)|\frac{dx}{x}=1.
\end{eqnarray*}
Hence $ \overline{J_{\alpha,n}}(x) \in X_c$ and therefore its Mellin transform is well defined. Now, we obtain 

\begin{eqnarray*}
[ \overline{J_{\alpha,n}}]^{\wedge}_M (c+iv)=\int_{0}^{\infty}  \overline{J_{\alpha,n}}(x)x^{c+iv}\frac{dx}{x}
=C_{\alpha,n}\int_{0}^{\infty}x^{iv}\frac{\textit{sin}^{2n} \left(\frac{\log x}{2 \alpha n} \right)}{\left(\frac{\log x}{2 \alpha n }\right)^{2n}}\frac{dx}{x}.
\end{eqnarray*}
On substituting $u=\log x,$ we obtain 
\begin{eqnarray*}
[ \overline{J_{\alpha,n}}]^{\wedge}_M (c+iv)&=&\frac{C_{\alpha,n}}{(n \alpha)^{2n}}\int_{-\infty}^{\infty}e^{iuv}\left(\widehat{\chi}_{[-\frac{1}{2n\alpha},\frac{1}{2n\alpha}]}(u)\right)^{2n}du\\
&=&\frac{C_{\alpha,n}}{(n \alpha)^{2n}}
\left({\chi}_{[-\frac{1}{2n\alpha},\frac{1}{2n\alpha}]}\ast {\chi}_{[-\frac{1}{2n\alpha},\frac{1}{2n\alpha}]}\ast...\ast {\chi}_{[-\frac{1}{2n\alpha},\frac{1}{2n\alpha}]}\right)(u), \ \ (2n\ \ times),
\end{eqnarray*}
where $\ast$ denotes the convolution. From the above relation, we see that $supp [ \overline{J_{\alpha,n}}]^{\wedge}_M\subseteq [-\frac{1}{\alpha},\frac{1}{\alpha}],$ hence we get $[ \overline{J_{\alpha,n}}]^{\wedge}_M (c+iv)=0,$ for $|v|>\frac{1}{\alpha}.$ Thus $ \overline{J_{\alpha,n}}$ is Mellin band-limited. Further, we observe that $[ \overline{J_{\alpha,n}}]^{\wedge}_M (0)=1$ and $[ \overline{J_{\alpha,n}}]^{\wedge}_M (2k\pi i)=0,$ for $k\neq 0.$ Hence using the Mellin-Poisson summation formula, we get 
\begin{eqnarray*}
\displaystyle\sum_{k=- \infty}^{\infty} \overline{J_{\alpha,n}}(e^{k}x)=\sum_{k=- \infty}^{\infty}[ \overline{J_{\alpha,n}}]^{\wedge}_M (2k\pi i)x^{-2k\pi i}=1.
\end{eqnarray*}
Hence $(\chi_{1})$ is satisfied.\\

In view of the Mellin transform, we obtain 
\begin{eqnarray*}
\frac{d}{dv}\left([\overline{J_{\alpha,n}}]^{\wedge}_M (0)\right)&=&\int_{0}^{\infty}  \overline{J_{\alpha,n}}(u)(i\log u)\frac{du}{u}= i C_{\alpha,n}\int_{0}^{\infty}\frac{\textit{sin}^{2n} \left(\frac{\log u}{2 \alpha n } \right)}{\left(\frac{\log u}{2 \alpha n }\right)}\log u \frac{du}{u}.
\end{eqnarray*}
Hence, we obtain $\dfrac{d}{dv}\left([{J_\alpha,n}]^{\wedge}_M (0)\right)=0.$ Thus from the Mellin-Poisson summation formula, we get $m_1( \overline{J_{\alpha,n}},u)=0.$ Therfore, the condition $(\chi_{2})$ is established.
Further, we see that $m_{1}^{ \overline{J_{\alpha,n}}} \neq -\dfrac{1}{2},$ the condition of Theorem \ref{xyz3} is fulfilled. \\

We now show that $(\chi_{4})$ is satisfied, i.e. 
$\displaystyle\lim_{w\rightarrow \infty} \sum_{|k-w\log x|>  w\gamma} | \overline{J_{\alpha,n}}(e^{-k} x^w)| \ |k- w\log x|=0 $ uniformly on $\mathbb{R}^{+}.$ Let $\epsilon>0$ and $n>1.$ Then, there exists $N \in \mathbb{Z}$ such that $\displaystyle \sum_{k>N} \frac{1}{k^{2n-1}} < \epsilon.$ For $w\gamma >N,$  we can write\\

\noindent 
$\displaystyle\sum_{|k-w\log x|>w\gamma} | \overline{J_{\alpha,n}}(e^{-k} x^w)| \ |k- w\log x|$
\begin{eqnarray*}
&=& \left( \sum_{k-w\log x >w\gamma} + \sum_{k-w\log x < -w\gamma} \right) | \overline{J_{\alpha,n}}(e^{-k} x^w)| \ |k- w\log x| \\
&:=& S_{1}+S_{2}.
\end{eqnarray*}
First we estimate $S_{1}.$ We have
\begin{eqnarray*}
S_{1} &=& \sum_{k-w\log x >w\gamma} C_{\alpha,2 n} sinc \left(\frac{|w\log x-k|}{2 \alpha n\pi}\right)^{2n} |k-w\log x| \\
& \leq & \sum_{k-w\log x >w\gamma} C_{\alpha,2 n} (2 \alpha n)^{2n} \frac{1}{|k-w\log x|^{2n-1}}\\
&\leq &  C_{\alpha,2 n} (2 \alpha n )^{2n} \sum_{k >N} \frac{1}{k^{2n-1}} 
< C_{\alpha,2 n} (2 \alpha n)^{2n} \epsilon. 
\end{eqnarray*}
Similarly, for $S_{2},$ we obtain 
\begin{eqnarray*}
S_{2}
& \leq & \sum_{k-w\log x <-w\gamma} C_{\alpha,2 n} (2 \alpha n )^{2n} \frac{1}{|k-w\log x|^{2n-1}}\\
&\leq &  C_{\alpha,2 n} (2 \alpha n )^{2n} \sum_{k=1}^{\infty} \frac{1}{(N+k)^{2n-1}} \\
&<& C_{\alpha,2 n} (2 \alpha n )^{2n} \epsilon. 
\end{eqnarray*}
On combining $S_1-S_2,$ we get $S < 2 C_{\alpha,2 n} (2 \alpha n)^{2n} \epsilon,$ for $n>1.$ Hence $(\chi_{4})$ is satisfied for $n>1.$\\

Now we establish the condition $(\chi_{3}).$
Let $u \in \mathbb{R}^+$ and $\beta < 2n -1, c=0.$ Then $\exists$ \ $k_{0} \in \mathbb{Z}$ such that $k_{0} \leq \log u < k_{0}+1.$ This gives
$|\log u-k| \geq |k-k_{0}|,$ if $k<k_{0}$ and $|\log u-k| > |k-(k_{0}+1)|,$ if $k>k_{0}+1.$ Now, we have 
\begin{eqnarray*}
\sum_{k=-\infty}^{\infty} |{ \overline{J_{\alpha,n}}}(e^{-k}u)| \ |k- \log u|^{\beta}  &=&
\left( \sum_{k <k_{0}} + \sum_{k=k_{0},k_{0}+1}+\sum_{k >k_{0}} \right) | \overline{J_{\alpha,n}}(e^{-k} u)| \ |k- \log u|^{\beta} \\
&:=& S_{1}^{'}+S_{2}^{'}+S_{3}{'}.
\end{eqnarray*}
Using definition of $ \overline{J_{\alpha,n}},$ $S_{1}^{'}$ is estimated by 
\begin{eqnarray*}
S_{1}^{'} &=& C_{\alpha,n} \sum_{k<k_{0}} \sin^{2n} \left( \frac{|k-\log u|}{2n \alpha} \right) (2\alpha n)^{2 n} \frac{1}{|\log u -k|^{2n-\beta}} \\
& \leq & C_{\alpha,n} (2\alpha n)^{2 n} \sum_{k <k_{0}} \frac{1}{|\log u -k|^{2n-\beta}}\\
& \leq & C_{\alpha,n} (2\alpha n)^{2 n} \sum_{k <k_{0}} \frac{1}{|k-k_0|^{2n-\beta}} \leq C_{\alpha,n} (2\alpha n)^{2 n} \sum_{k=1}^{\infty} \frac{1}{k^{2n-\beta}}.
\end{eqnarray*}
The above sum is finite if $\beta < 2n-1.$ Similarly, for $S_3^{'},$ we obtain 
\begin{eqnarray*}
S_3^{'}
& \leq & C_{\alpha,n} (2\alpha n)^{2 n} \sum_{k >k_{0}+1} \frac{1}{|\log u -k|^{2n-\beta}}\\
& \leq & C_{\alpha,n} (2\alpha n)^{2 n} \sum_{k >k_{0}+1} \frac{1}{|k-(k_0+1)|^{2n-\beta}}
 \leq C_{\alpha,n} (2\alpha n)^{2 n} \sum_{k=1}^{\infty} \frac{1}{k^{2n-\beta}}.
\end{eqnarray*}
This gives $S_3^{'} < \infty$ for $\beta < 2n-1.$ Finally $S_2^{'}$ is estimated as 
\begin{eqnarray*}
S_2^{'} &=& C_{\alpha,n} \left( \frac{\sin(\frac{|\log u -k_0|}{2n \alpha})}{\frac{|\log u -k_0|}{2n \alpha}} \right)^{2n} |\log u - k_0|^{\beta} + C_{\alpha,n} \left( \frac{\sin(\frac{|\log u -k_0-1|}{2n \alpha})}{\frac{|\log u -k_0-1|}{2n \alpha}} \right)^{2n} |\log u - (k_0-1)|^{\beta} \\
& \leq & C_{\alpha,n} \left( \frac{\sin(\frac{|\log u -k_0|}{2n \alpha})}{\frac{|\log u -k_0|}{2n \alpha}} \right)^{2n} + C_{\alpha,n} \left( \frac{\sin(\frac{|\log u -k_0-1|}{2n \alpha})}{\frac{|\log u -k_0-1|}{2n \alpha}} \right)^{2n} \\
&\leq & 2 C_{\alpha,n} \sup_{0 \leq u \leq 1} \left( \frac{\sin(\frac{u}{2n \alpha})}{\frac{u}{2n \alpha}} \right)^{2n} 
\leq  2 C_{\alpha,n}.
\end{eqnarray*}
Combining the estimates $S_1^{'},S_2^{'},S_3^{'},$ we get 
$$\sup_{u \in \mathbb{R}^+} \sum_{k=-\infty}^{\infty} |{ \overline{J_{\alpha,n}}}(e^{-k}u)| \ |k- \log u|^{\beta} < \infty$$ and hence $(\chi_{3})$ is verified. 

Putting $c=0$ in the Mellin-Jackson kernel $ \overline{J_{\alpha,n}}$ and taking the Mellin derivative of the kernel, we get 

$$(\theta  \overline{J_{\alpha,n}})(x)= \frac{ C_{\alpha,n}}{\log^{2n}x} \left[\frac{2 n}{\rho^{2n-1}} \sin^{2n-1} (\rho \log x ) \left( \cos(\rho \log x)-\frac{\sin(\rho \log x)}{\rho \log x} \right)\right].$$ 
Proceeding along the lines of the estimate of $S_1^{'},S_2^{'},S_3^{'}$, it follows that $M_{1}(\theta \overline{J_{\alpha, n}}) < \infty,$ and thus 
$\overline{J_{\alpha,n}}$ satisfies the assumptions of Theorem \ref{xyz5} for $n>1$ and $\beta<2n-1$.   

\section{Acknowledgement}
\noindent 
S. Bajpeyi gratefully thank Indian Institute of Science Education and Research (IISER) Thiruvananthapuram for the postdoctoral fellowship to carry out this research work. A. Sathish Kumar acknowledges DST-SERB, India Research Grant MTR/2021/000428 for the financial support and NFIG Grant, IIT Madras, Grant No. RF/22-23/0984/MA/NFIG/009017.

\end{document}